\documentclass[a4paper,12pt, reqno]{amsart}

\usepackage[ansinew]{inputenc}
\usepackage{amsmath,amsthm,amsfonts,amssymb,stackengine}
\usepackage{mathrsfs}
\usepackage[left=2.50cm,right=2.25cm,top=2.50cm,bottom=2.50cm]{geometry}
\usepackage{graphicx}
\usepackage[english]{babel}
\usepackage{hyperref}
\usepackage{enumitem,fancyhdr}
\usepackage{etoolbox,lineno}

\usepackage{fouridx}

\usepackage{soul}
\usepackage[dvipsnames]{xcolor}

\newcommand\red[1]{{\color{red}#1}}
\newcommand\blue[1]{{\color{blue}#1}}

\usepackage{mathtools} 
\numberwithin{equation}{section}
\theoremstyle{plain}
\newtheorem{Thm}{Theorem}[section]

\newtheorem{Lem}[Thm]{Lemma}

\theoremstyle{definition}
\newtheorem{Def}[Thm]{Definition}

\usepackage{parskip}
\usepackage{marginnote}
\usepackage{changes}
\definechangesauthor[name=tesfa, color=red]{tesfa}

\begin{document}
	\title{A Front Fixing Crank-Nicolson Finite Deference for the American Put Options Model}
	\author{Z.I. Ali}
	\address{Z.I. Ali, Department of Mathematical Sciences,
		University of South Africa,  Florida 0003, South Africa.}
	\email{alizi@unisa.ac.za}
	\author{M.A. Abebe}
	\address{M.A. Abebe, Department of Mathematical Sciences,
	University of South Africa,  Florida 0003, South Africa.}
	\email{67121497@mylife.unisa.ac.za}

	\maketitle
	
	\begin{abstract}
In this paper, we present a novel approach to solving the American put options pricing model by hugely relying on a front-fixing Crank-Nicolson finite difference method. Since the American put option pricing model is a widely used financial model for valuing an option with the right to sell an underlying asset at a fated price which generally decided in advance. The method we proposed here, solves the problem of early exercise by introducing a front-fixing technique that permits for efficient and accurate valuation of an American put option. As in the comparison to other approaches in the existing literature, we can assert that this method is stable, accurate, consistent, and efficient. The results that we obtained here from the numerical experiments demonstrate not only the efficacy of the proposed method but also in consistently and accurately pricing American put options with a stable scheme. Under some appropriate conditions on the step size discretization, we also show the positivity and monotonicity of the coefficient involved in the numerical scheme used. 
	\end{abstract}
	
	\vspace*{1cm}
	
\textbf{Keywords}: New Crank-Nicolson, American put option, finite difference, discretization, front-fixing method, numerical scheme, efficiency-accuracy, consistency and stability of front-fixing, positivity, monotonicity.
\newpage
\hrule

\vspace*{1cm}

\tableofcontents

\newpage

\pagestyle{fancy}
\fancyhead{}\fancyfoot{}

\lhead{Z.I. \& M.A.}
\rhead{American Put Options Model}
\cfoot{\arabic{page}}

\section{Introduction}
The American put options model is a widely used financial tool for valuing derivative contracts, allowing investors the flexibility to exercise their options at any time prior to expiration, see \textit{i.e.}, Riccardo et al. \cite{fazio2021front} and Company et al.\cite{egorova2014solving} for more details on this topic. We must emphasize that accurate and efficient numerical methods are crucial for pricing these options, and among others, one popular approach is the Crank-Nicolson finite difference scheme. \\

In this study, we introduce a novel variant of the Crank-Nicolson method known as the Front Fixing Crank-Nicolson Finite Difference method which aims to address the challenges posed by the early exercise feature of American put options, where the optimal exercise boundary is not known in advance, we refer to the work  \blue{\cite{holmes2012front,tangman2008fast}} for further details in this direction. By incorporating front fixing techniques into the Crank-Nicolson scheme, we propose an improved numerical method that may capture any behavior of the early exercise much more effectively.

In this paper, we present a detailed description of the Front Fixing Crank-Nicolson Finite Difference method and explore its advantages in pricing American put options. Further, we examine its stability, accuracy, and convergence properties through numerical experiments and we then compare its performance with other existing methods. Furthermore, we discuss the implications and potential applications of the Front Fixing approach in the broader context of option pricing which literally be used by other authors for risk management.

We can assert that this research contributes to the ongoing efforts and the existing literature in developing robust and efficient numerical schemes for pricing American put options, which will assist us to provide valuable insights for financial practitioners, researchers and investors involved in option valuation and as well as portfolio management.\\

American put options have additional privileges for the holder which European options may not possess; the holder is gifted extra right exercised at any moment before the expiration time. This privilege makes the American put option the majority traded in the options market; worth more than its counterpart and pricing has a significant role in the derivatives market as well as in the derivatives market. We refer for instance to Riccardo et.al \cite{fazio2021front} for additional information in this direction. It is worth to note that the pricing option using PDE with some strict assumptions primarily developed by Black and Scholes and  Merton in \cite{RePEc:ucp:jpolec:v:81:y:1973:i:3:p:637-54}   \cite{merton1973theory} respectively, where further references can be found therein. 
Schwartz \cite{schwartz1977valuation}, Brennan, and Schwartz\cite{brennan1977valuation,brennan1978finite} were the pioneer for solving APO by applying the finite difference method, and its accuracy has been proved by Jaillet et al.\cite{jaillet1990variational}. Company et.al \cite{egorova2014solving}, Riccardo et al. \cite{fazio2021front}, H. Han and X. Wu \cite{han2003fast}, Pantazopoulos et al. \cite{pantazopoulos1998front} and Tangman et al.\cite{tangman2008fast} have considered the front tracking methods for  keeping  track of the boundaries as well as the discretization of the problem by changing domain. The hassle in pricing American options arises from determining the optimal exercise of the underlying asset. The optimal exercise boundary is determined by analyzing how the boundary changes as the option near maturity. Note that this boundary converts the free boundary problem to a semi-infinite domain and minimizes the complexity of the non-linear model.

By applying a change in variables, the front-fixing method maps free boundary variables into a fixed domain, which reduces the difficulty  see \textit{e.g.}, \cite{holmes2012front,tangman2008fast,wu1997front,zhu2003application} for more information and the references therein.

Fractional calculus simplifies the modeling of a complex system in many disciplines, and its role has increased in the last decade in many fields such as finance and economics for more details we refer to \cite{kumar2014numerical, yavuz2018different, nuugulu2021robust}; in science and engineering  \cite{baleanu2010new,qureshi2019new,atangana2016new} among others. 

The valuation of APO contracts provoked the interest of both financial engineers and mathematicians over the last few decades due to its privilege and being the most traded in the options market. In this article, we implement Caputo Fabrizio fractional derivatives in the classical Black-Scholes equation to present the Crank-Nicolson finite difference scheme with FFM for solving APO. Here, we present a new scheme of the APO developed by using the Crank-Nicolson implicit finite difference scheme and FFM without dividend payments. This will transfer the variable domain problem to the fixed rectangular domain of the non-linear problem using a front-fixing transformation.
\\

Asymptotic boundary conditions are imposed on a non-linear problem by the appropriate selection of a new truncated boundary. Then, we convert the problem to a new non-linear PDE with a free boundary, and in the computation of the solution, the free boundary will be taken as a new variable. Finally, we create a Crank-Nicolson finite difference scheme and evaluate its consistency and stability.\\

This manuscript is structured as follows: Section 2 deals with the basic concepts such as definition of Caputo-Fabrizio fractional derivatives of APO with its fixed-domain transformation and then we introduce the Crank-Nicolson discretization of the transformed equation. Section 3 addresses the numerical solution positivity and monotonicity w.r.t. time at the optimal exercise boundary. Section 4 discusses the stability and consistency of the scheme. 


\section {The Front-Fixing Method and Discretization}
\subsection{The American Put Options Model}
Let's assume that an underlying asset has a price of  $X$ at time $t$. Here, we considered the following mathematical model for the value of an APO to sell the asset $V(X,t)$.\\
\begin{align}
&\frac{\partial V}{\partial t}-\frac{1}{2}\sigma^2 X^2\frac{\partial^2 V}{\partial^2 X}-rX\frac{\partial V}{\partial X} +rV=0, \text{ for } X=X^*(t), \text{ and }  t\in [0, T],\label{Eeq1}\\
& V(X,T)= \max\{ E-X, 0 \}, \text{ for } X \geq 0,\\
&\frac{\partial V}{\partial X}(X^*(t), t)=-1,\\
&V(X^*(t),t)=E-X^*(t),\\
& \lim\limits_{X\to \infty}V(X, T)=0,\\
& X^*(T)=E,\\
&V(X(t), t) = E-X,\text{ for } X\in [0, X^*(t)).\label{eq1}
\end{align}
Where\\
$\tau=T-t$ represents the period until maturity $T$;\\
$ X^*(t)$ unknown early exercise free boundary with $X^*(t) > 0$;\\
$E$, is strike price;\\
T expiry time;\\
 $ \sigma $ volatility  and ;\\
$r$ the interest rate (risk-free).\\
The Black-Scholes-Merton equation, \eqref{Eeq1}-\eqref{eq1} was developed by  Black and Scholes \cite{black1973pricing}, and Merton\cite{merton1973theory}.
\subsection{The Front-Fixing Method}
A variety of free boundary problems that arise in physics are solved using the FFM, for detail refer \cite{crank1984free}. Wu and Kwok \cite{wu1997front} pioneered in applying and solving the free boundary option pricing using FFM.The front-fixing method has an advantage in direct computing the free boundary problem. That means when we directly apply finite differences and finite element methods to free boundary problems, like the equation from \eqref{Eeq1} to \eqref{eq1}, it is hard to manage the computational mesh points or elements and the problem solved by this technique. According to Landau \cite{landau1950heat}, the difficulty is removed by transforming unknown and time-varying boundaries into known and fixed lines.
The method's fundamental principle is to modify the original problem into non-linear PDE by transforming variables from a free boundary to a bounded domain. First, we rewrite equations \eqref{Eeq1} to \eqref{eq1} over the bounded domain as new parabolic equations then introduce a truncated boundary, and finally apply finite difference schemes to solve approximated problems across a bounded domain.In this paper, we can apply this technique to the American put option problem ( \eqref{Eeq1} to \eqref{eq1} ) with Crank-Nicolson discretization. We consider dimensionless new variables similar to Wu and Kwok \cite{wu1997front} and Riccardo. F et.al.\cite{fazio2021front} did. To conclude this we define
\begin{align}
& y = \ln{\frac{X}{X^*(\tau)}}, \label{eq2}\\
& X_f(\tau) =\frac{X^*(\tau)}{E}, \label{eq2-2} \\
& v(y,\tau) =\frac{V(X,\tau)}{E} \label{eq2-3}
\end{align}
Where $X_f(\tau)$  is the normalized function and it is mapped on the fixed line $y = 0$, $0\leq v(y,\tau)\leq 1$ and $0\leq X_f(\tau)\leq 1$. Then, by plugging \eqref{eq2}-\eqref{eq2-3} into \eqref{Eeq1}-\eqref{eq1} become
\begin{align}
\begin{split}
&\frac{\partial v}{\partial \tau}-\frac{1}{2}\sigma^2 \frac{\partial^2 v}{\partial^2 y}-(r-\frac{\sigma^2}{2})\frac{\partial v}{\partial y}-\frac{1}{X_f(\tau)}\frac{dX_f(\tau)}{d\tau}\frac{\partial v}{\partial y}+rv=0,\\ &\text{   For  } X=X^*(t),   t\in [0, T] \label{Eeq3}
\end{split} 
		\\
\begin{split}
&v(y;0)=0, \text{ for } 0\leq y,    X_f(0)=1,
\end{split}
	\\
\begin{split}
&\lim\limits_{y\to \infty}v(y,\tau)=0.
\end{split}
	 \\
\begin{split}
&v(0;\tau)=1-X_f(\tau),\frac{\partial v}{\partial y}(0,\tau)=-X_f(\tau)
\text{  on  } 0 \leq t \leq T \text{  and }  0 \leq x \leq \infty \label{eq3}
\end{split}
\end{align}
Now we replace \eqref {Eeq3}  by Caputo-Fabrizio time-fractional derivative and equation  \eqref{Eeq3}-\eqref{eq3} become
\begin{align}
\begin{split}
&^c_0D^\alpha _t v(\tau,y) -\frac{1}{2}\sigma^2 \frac{\partial^2 v}{\partial^2 y}-(r-\frac{\sigma^2}{2})\frac{\partial v}{\partial y}-\frac{1}{X_f(\tau)}\frac{dX_f(\tau)}{d\tau}\frac{\partial v}{\partial y}+rv=0, \\& \text{  For  }X=X^*(t),\text{and~} t\in [0,T]\label{eq4}
\end{split}
\\
\begin{split}
v(0,y)=0, \text{ for  } 0\leq y,X_f(0)=1,
\end{split} \label{eq5} \\
\begin{split}
\lim\limits_{y\to \infty}v(\tau,y)=0,
\end{split}\label{eq6} \\
\begin{split}
 &v(\tau,0)=1-X_f(\tau),	\frac{\partial v}{\partial y}(\tau,0)=-X_f(\tau)
\\& \textup{ On the outlined domain, } 0 \leq t \leq T  \text{ and } 0\leq y \leq \infty. \label{eq7}
\end{split}
\end{align}
Where 
\begin{align}
^c_0D^\alpha _t v(\tau,y)=\frac{1}{1-\alpha}\int_{0}^{t}e^\frac{-\alpha(t-\tau)}{(1-\alpha)}v'(\tau)d\tau 
\end{align}
which is the formula for the new fractional Caputo-Fabrizio derivative of order $0<\alpha<1$ defined by Losada and Nieto\cite{losada2015properties}and its Laplace transformation  given by Itrat.A and Dumitru.V  \cite{Mirza2017FundamentalST} as
\begin{align}
L(^c_0D^\alpha _t v(\tau,y))=\frac{SLv(t)-v(0)}{(1-\alpha)s+\alpha} 
\end{align}
So far, this chapter has focused on creating a parabolic equation from the free boundary problem over the bounded domain and converting it to a time-fractional Caputo-Fabrizio derivatives equation. The succeeding sections will define the size mesh of grid points in time and space.

\subsection{The Crank-Nicolson Discretization}
In this section, we can define a numerical scheme to compute the grid values of APO. For this purpose, we start by truncating 
$(0, \infty)$  into  $(0, Y)$ to solve the problem \eqref{Eeq3}-\eqref{eq3}. According to Wilmott et.al\cite{wilmott1993option}, it was logical to estimate and set the upper bound by $Y=4E$. Hence,we obtain $v(Y,t )=0$ from the boundary condition is given by  \eqref{eq3}. In the rest of the paper, we shall solve PDE using the Crank-Nicolson finite difference method.
That is, the problem \eqref{Eeq3}-\eqref{eq3} can be studied on $[0, Y] \times (0,T]$ and let $ M \in \mathbb{N}$  and  $ \mu \in \mathbb{R^{+}}$ ; now we set step-size as follow
\begin{align*}
&\Delta y=\frac{Y}{M}=\frac{4E}{M}\\
&\Delta \tau={\mu}{(\Delta y)^2},\\
&\textup{The integer } N=\lceil\frac{T}{\Delta\tau}\rceil
\end{align*}
Where $\lceil.\rceil:\mathcal{R^+}\rightarrow \mathbb{N}$  is the  ceiling function from a positive real number to the least of a natural number greater or equal to it. Hence, grid-ratio $\mu$ is given by
\begin{align*}
\begin{split}
\mu=\frac{\Delta \tau}{(\Delta y)^2}
\end{split}
\end{align*}
A mesh of grid points can be introduced in the finite domain as follows
\begin{align*}
& y_m = m\Delta y,\\
& \tau^n= n\Delta\tau ,\text{  for  } m = 0, 1, \ldots, M \text{  and  } n = 0,1,\ldots,N.  
\end{align*}
To compute the grid values, we construct a numerical scheme and the free boundary values
\begin{align*}
\begin{split}
v^m_n\approx v(\tau^n,y_m)\\
X^n_f=X_f(\tau^n)
\end{split}
\end{align*}
for $m = 0, 1, \ldots, M$ and $n = 0, 1, \ldots, N$.\\
Having the size of the mesh of grids points concerning time and space which we can use for computing a numerical scheme of the grid value of APO, we will now move to develop a new Crank-Nicolson finite difference scheme of APO.

\subsection {A New Crank-Nicolson  Finite Difference Scheme}
In this part, a Crank-Nicolson finite difference scheme is presented.Recall from  \eqref{eq4} to \eqref{eq7} we have that\\
\begin{align}
\begin{split}
&^c_0D^\alpha _t v(\tau,y)=\frac{1}{2}\sigma^2 \frac{\partial^2 v}{\partial^2 y}+(r-\frac{\sigma^2}{2})\frac{\partial v}{\partial y}+\frac{1}{X_f(\tau)}\frac{dX_f(\tau)}{d\tau}\frac{\partial v}{\partial y}-rv=0,\\& \text{ for }   X=X^*(t), \text{ and } t\in [0,T],
\end{split} \label{eq8} \\
\begin{split}
v(0,y)=0,\text{ for } y\geq 0,X_f(0)=1,
\end{split}\label{eq9}	\\
\begin{split}
\lim\limits_{y\to \infty}v(\tau,y)=0,
\end{split} \label{eq10}\\
\begin{split}
&v(\tau,0)=1-X_f(\tau),\frac{\partial v}{\partial y}(\tau,0)=-X_f(\tau) \\&\textup{  on  } 0 \leq t \leq T \text{  and   }  0\leq t \leq \infty.\label{eq11}
\end{split}
\end{align}
First, discretize the left-hand side term of \eqref{eq8}. According to Qureshi et al.\cite{qureshi2019new} the discretization scheme of the term is given by\\ 
\begin{align}
\begin{split}
^c_0D^\alpha _t v(\tau=t_n,y=y_m) =\frac{1}{\Delta\tau\alpha}(e^{\frac{\alpha}{1-\alpha}\Delta\tau} -1) \sum_{k=1}^{n}(v(t_{n+1-k},y_m)-v(t_{n-k},y_m))e^{\frac{-\alpha}{1-\alpha}k\Delta\tau} \label{eq12}
\end{split}
\end{align}
with
\begin{align}
 \mathcal{TD}^{m}_{n}=\mathcal{O}(\Delta\tau)
\end{align}
Where $\mathcal{TD}^{m}_{n}$ is the local truncation error and independent of the fractional order $\alpha$ for detail see \cite{qureshi2019new}.\\
Next, we can find the Crank-Nicolson discretization of the latter three components of the right-hand side equation \eqref{eq8}, ignoring their local truncation error, with $ m= 1, 2,\hdots, M-1 $ given by\\
\begin{align}
&\frac{1}{2}\sigma^2 \frac{\partial^2 v}{\partial^2 y}+(r-\frac{\sigma^2}{2})\frac{\partial v}{\partial y}+\frac{1}{X_f(\tau)}\frac{dX_f(\tau)}{d\tau}\frac{\partial v}{\partial y}-rv \notag \\& =\frac{1}{2}(\frac{1}{2}\sigma^2 \frac{v^{m+1}_{n+1}-2v^{m}_{n+1}+v^{m-1}_{n+1}}{(\Delta y)^2} +(r-\frac{\sigma^2}{2})\frac{v^{m+1}_{n+1}-v^{m-1}_{n+1}}{2\Delta y}-rv^{m}_{n+1}\notag\\&+\frac{X^{n+1}_f(\tau)-X^{n}_f(\tau)}{\Delta\tau X^{n}_f(\tau)}\frac{v^{m+1}_{n+1}-v^{m-1}_{n+1}}{2\Delta y}+\frac{1}{2}\sigma^2 \frac{v^{m+1}_{n}-2v^{m}_{n}+v^{m-1}_{n}}{(\Delta y)^2}\notag\\&+(r-\frac{\sigma^2}{2})\frac{v^{m+1}_{n}-v^{m-1}_{n}}{2\Delta y}+\frac{X^{n+1}_f(\tau)-X^{n}_f(\tau)}{\Delta\tau X^{n}_f(\tau)}\frac{v^{m+1}_{n}-v^{m-1}_{n}}{2\Delta y}-rv^{m}_{n})\label{eq13}
\end{align}
Then, we can obtain a new Crank-Nicolson time scheme for internal special nodes by inserting equations \eqref{eq12} and \eqref{eq13} as shown below
\begin{align}
&\frac{1}{\Delta\tau\alpha}(e^{\frac{\alpha}{1-\alpha}\Delta\tau} -1) \sum_{k=1}^{n}(v(t_{n+1-k},y_m)-v(t_{n-k},y_m))e^{\frac{-\alpha}{1-\alpha}k\Delta\tau} \notag \\&=
\frac{1}{2}(\frac{1}{2}\sigma^2 \frac{v^{m+1}_{n+1}-2v^{m}_{n+1}+v^{m-1}_{n+1}}{(\Delta y)^2} +(r-\frac{\sigma^2}{2})\frac{v^{m+1}_{n+1}-v^{m-1}_{n+1}}{2\Delta y}-rv^{m}_{n+1}\notag \\&+\frac{X^{n+1}_f(\tau)-X^{n}_f(\tau)}{\Delta\tau X^{n}_f(\tau)}\frac{v^{m+1}_{n+1}-v^{m-1}_{n+1}}{2\Delta y}+\frac{1}{2}\sigma^2 \frac{v^{m+1}_{n}-2v^{m}_{n}+v^{m-1}_{n}}{(\Delta y)^2}\notag \\&+(r-\frac{\sigma^2}{2})\frac{v^{m+1}_{n}-v^{m-1}_{n}}{2\Delta y}+\frac{X^{n+1}_f(\tau)-X^{n}_f(\tau)}{\Delta\tau X^{n}_f(\tau)}\frac{v^{m+1}_{n}-v^{m-1}_{n}}{2\Delta y}-rv^{m}_{n})\label{13a}
\end{align}
Equation \blue{\eqref{13a}} rewritten as
\begin{align}
 &\sum_{k=1}^{n}(v(t_{n+1-k},y_m)-v(t_{n-k},y_m))e^{\frac{-\alpha}{1-\alpha}k\Delta\tau}\notag \\& =A_{n} v^{m+1}_{n+1}+ B v^{m}_{n+1}+ C_{n} v^{m-1}_{n+1}+ A_{n} v^{m+1}_{n}+ B v^{m}_{n} +C_{n} v^{m-1}_{n} \label{13b}
\end{align}
Where
\begin{align*}
&A_{n}=\frac{\Delta\tau\alpha}{e^{\frac{\alpha}{1-\alpha}\Delta\tau} -1}(\frac{1}{4(\Delta y)^2}\sigma^2 +\frac{1}{4\Delta y}(r-\frac{\sigma^2}{2}) +\frac{X^{n+1}_f(\tau)- X^{n}_f(\tau)}{4\Delta y\Delta\tau X^{n}_f(\tau)}) \notag \\
&C_{n}=\frac{\Delta\tau\alpha}{e^{\frac{\alpha}{1-\alpha}\Delta\tau} -1}(\frac{1}{4(\Delta y)^2}\sigma^2 - \frac{1}{4\Delta y}(r-\frac{\sigma^2}{2}) -\frac{X^{n+1}_f(\tau)- X^{n}_f(\tau)}{4\Delta y\Delta\tau X^{n}_f(\tau)}) \text{  and  }\\ 
 &B=\frac{-\Delta\tau\alpha}{2(e^{\frac{\alpha}{1-\alpha}\Delta\tau} -1)}(\frac{\sigma^2}{(\Delta y)^2}+r). 
  \end{align*}  
We intend to represent and verify equation \eqref{13b} in a matrix form. For this purpose, we intend to use a sort of induction technique. To show this, we let $n=1$, with this value of $n$, we have that
\begin{align*}
(v^{m}_{1}-v^{m}_{0})e^{\frac{-\alpha}{1-\alpha}\Delta\tau}
=A_{1}v^{m+1}_{2}+B v^{m}_{2}+C_{1}v^{m-1}_{2}+A_{1}v^{m+1}_{1}+B v^{m}_{1}+C_{1} v^{m-1}_{1}
\end{align*}
which implies that
\begin{align*}
v^{m}_{1}e^{\frac{-\alpha}{1-\alpha}\Delta\tau}
=A_{1}v^{m+1}_{2}+B v^{m}_{2}+C_{1} v^{m-1}_{2}+ A_{1}v^{m+1}_{1}+B v^{m}_{1}+ C_{1}v^{m-1}_{1}
\end{align*}
from which we infer that
\begin{align}
\begin{split}
A_{1}v^{m+1}_{2}+B v^{m}_{2}+ C_{1}v^{m-1}_{2}=-A_{1} v^{m+1}_{1} - (B-e^{\frac{-\alpha}{1-\alpha}\Delta\tau})v^{m}_{1}-C_{1} v^{m-1}_{1}.
\end{split}
\end{align}
Let $n=2$
\begin{align*}
(v^{m}_{2}-v^{m}_{1})e^{\frac{-\alpha}{1-\alpha}\Delta\tau}+ (v^{m}_{1}-v^{m}_{0})e^{\frac{-\alpha}{1-\alpha}2\Delta\tau}
 =A_{2} v^{m+1}_{3}+B v^{m}_{3}+C_{2}v^{m-1}_{3}+A_{2}v^{m+1}_{2}+B v^{m}_{2}+C_{2} v^{m-1}_{2} \\  
\end{align*}
which implies that
\begin{align*}
(v^{m}_{2}-v^{m}_{1})e^{\frac{-\alpha}{1-\alpha}\Delta\tau}+v^{m}_{1}e^{\frac{-\alpha}{1-\alpha}2\Delta\tau}
 =A_{2} v^{m+1}_{3}+B v^{m}_{3}+C_{2}v^{m-1}_{3}+A_{2}v^{m+1}_{2}+B v^{m}_{2}+C_{2} v^{m-1}_{2}\\ 
\end{align*}
from which we infer that
\begin{align}
 A_{2}v^{m+1}_{3}+B v^{m}_{3}+C_{2}v^{m-1}_{3}=-A_{2}v^{m+1}_{2}-(B-e^{\frac{-\alpha}{1-\alpha}\Delta\tau})v^{m}_{2}-C_{2}v^{m-1}_{2}-v^{m}_{1}(e^{\frac{-\alpha}{1-\alpha}2\Delta\tau}-e^{\frac{-\alpha}{1-\alpha}\Delta\tau})   
\end{align}    
Let $n=3$
\begin{align*}
&(v^{m}_{3}-v^{m}_{2})e^{\frac{-\alpha}{1-\alpha}\Delta\tau}+(v^{m}_{2}-v^{m}_{1})e^{\frac{-\alpha}{1-\alpha}2\Delta\tau}+ (v^{m}_{1}-v^{m}_{0})e^{\frac{-\alpha}{1-\alpha}3\Delta\tau}=\notag \\&A_{3}v^{m+1}_{4}+B v^{m}_{4}+C_{3}v^{m-1}_{4} +A_{3}v^{m+1}_{3}+B_{3}v^{m}_{3}+C_{3}v^{m-1}_{3} \\     
\end{align*}
which implies that
\begin{align*}
  (v^{m}_{3}-v^{m}_{2})e^{\frac{-\alpha}{1-\alpha}\Delta\tau}+(v^{m}_{2}-v^{m}_{1})e^{\frac{-\alpha}{1-\alpha}2\Delta\tau}+ v^{m}_{1}e^{\frac{-\alpha}{1-\alpha}3\Delta\tau} =\notag \\ A_{3}v^{m+1}_{4}+B v^{m}_{4}+C_{3}v^{m-1}_{4}+ A_{3}v^{m+1}_{3}+B_{3}v^{m}_{3}+C_{3}v^{m-1}_{3}  
\end{align*}
from which we infer that
\begin{align}
A_{3}v^{m+1}_{4}+B v^{m}_{4}+C_{3}v^{m-1}_{4}
 =&-A_{3}v^{m+1}_{3}-(B-e^{\frac{-\alpha}{1-\alpha}\Delta\tau})v^{m}_{3}- C_{3} v^{m-1}_{3} \notag\\&- v^{m}_{2}(e^{\frac{-\alpha}{1-\alpha}\Delta\tau}-e^{\frac{-\alpha}{1-\alpha}2\Delta\tau})+v^{m}_{1}(e^{\frac{-\alpha}{1-\alpha}2\Delta\tau}-e^{\frac{-\alpha}{1-\alpha}3\Delta\tau})  
\end{align}
Let $n=4$
\begin{align*}
&(v^{m}_{4}-v^{m}_{3})e^{\frac{-\alpha}{1-\alpha}\Delta\tau}+(v^{m}_{3}-v^{m}_{2})e^{\frac{-\alpha}{1-\alpha}2\Delta\tau}+(v^{m}_{2}-v^{m}_{1})e^{\frac{-\alpha}{1-\alpha}3\Delta\tau}+ (v^{m}_{1}-v^{m}_{0})e^{\frac{-\alpha}{1-\alpha}4\Delta\tau}\notag \\
&=A_{4}v^{m+1}_{5}+B v^{m}_{5}+C_{4}v^{m-1}_{5}+ A_{4}v^{m+1}_{4}+   B v^{m}_{4} + C_{4}v^{m-1}_{4}\\  
\end{align*}
from which we infer that
\begin{align}
A_{4}v^{m+1}_{5}+B v^{m}_{5}+C_{4}v^{m-1}_{5}= &-A_{4}v^{m+1}_{4}-(B-e^{\frac{-\alpha}{1-\alpha}\Delta\tau})v^{m}_{4}-C_{4}v^{m-1}_{4} +(e^{\frac{-\alpha}{1-\alpha}\Delta\tau}-e^{\frac{-\alpha}{1-\alpha}2\Delta\tau})v^{m}_{3}\notag \\&+(e^{\frac{-\alpha}{1-\alpha}2\Delta\tau}-e^{\frac{-\alpha}{1-\alpha}3\Delta\tau})v^{m}_{2} +(e^{\frac{-\alpha}{1-\alpha}3\Delta\tau}-e^{\frac{-\alpha}{1-\alpha}4\Delta\tau})v^{m}_{1} 
\end{align}
We have a system of linear equations at each time step with $M-1$ variables and $M+1$ equations. The equations only hold for $m = 1, 2, \hdots, M-1$ and the boundary condition applies again for the missing two equations. The matrix form of the equation is given by  
\[
\begin{bmatrix} 
A_{1} & B & C_{1}&0\hdots &\hdots&\hdots &\hdots \\
 0& A_{2} & B & C_{2}& 0\hdots &\hdots & \hdots &\hdots\\
0 & 0 & A_{3} & B & C_{3}&0 &\hdots &\hdots\\
0 & 0 & 0 & A_{4}& B & C_{4}&0\hdots\\
\vdots &\vdots &\vdots& \vdots & \vdots & \ddots & \ddots & \ddots\\
0 &0 &0  & 0 & 0& A_{n+1}& B & C_{n+1}\\ 
\end{bmatrix}
\times
\begin{bmatrix}
 v^{0}_{n+1}\\
v^{1}_{n+1}\\ 
v^{2}_{n+1}\\
\vdots\\ 
\vdots\\ 
v^{M}_{n+1} 
\end{bmatrix}\\
=\]
\[\begin{bmatrix} 
A_{1} & B & C_{1}&0\hdots &\hdots&\hdots &\hdots \\
 0& A_{2} & B & C_{2}& 0\hdots &\hdots & \hdots &\hdots\\
0 & 0 & A_{3} & B & C_{3}&0 &\hdots &\hdots\\
0 & 0 & 0 & A_{4}& B & C_{4}&0\hdots\\
\vdots &\vdots &\vdots& \vdots & \vdots & \ddots & \ddots & \ddots\\
0 &0 &0  & 0 & 0& A_{n+1}& B & C_{n+1}\\ 
\end{bmatrix}
\times
\begin{bmatrix}
 v^{0}_{n}\\
v^{1}_{n}\\ 
v^{2}_{n}\\
\vdots\\ 
\vdots\\ 
v^{M}_{n} 
\end{bmatrix}\\
+
\begin{bmatrix} 
\ v^0_0 \\
\sum_{k=1}^{n}(e^{\frac{-\alpha}{1-\alpha}(n+1-k)\Delta\tau}-e^{\frac{-\alpha}{1-\alpha}(n-k)\Delta\tau})v^{1}_{k}\\
 \sum_{k=1}^{n}(e^{\frac{-\alpha}{1-\alpha}(n+1-k)\Delta\tau}-e^{\frac{-\alpha}{1-\alpha}(n-k)\Delta\tau})v^{2}_{k} \\
\sum_{k=1}^{n}(e^{\frac{-\alpha}{1-\alpha}(n+1-k)\Delta\tau}-e^{\frac{-\alpha}{1-\alpha}(n-k)\Delta\tau})v^{3}_{k}\\
\sum_{k=1}^{n}(e^{\frac{-\alpha}{1-\alpha}(n+1-k)\Delta\tau}-e^{\frac{-\alpha}{1-\alpha}(n-k)\Delta\tau})v^{4}_{k} \\
\vdots \\
\sum_{k=1}^{n}(e^{\frac{-\alpha}{1-\alpha}(n+1-k)\Delta\tau}-e^{\frac{-\alpha}{1-\alpha}(n-k)\Delta\tau})v^{M-1}_{k}\\ 
\end{bmatrix}
\]
The above matrix has $M-1$ rows and $M+1$ columns, which is the representation of the $M-1$ equation and $M+1$ unknowns. Since, $V^{-1}_{N+1}$ and $V^{M+1}_{N+1}$ are out of domain.The two equations that we are missing come from the boundary conditions. Using these conditions, we are going to convert these systems of equations into systems of equations involving square matrices. The aim is to write the systems of equations in the form of square matrices, where the details of the boundary conditions are fully incorporated.
From the initial condition \eqref{eq9} and boundary condition \eqref{eq10} we get 
\begin{align*}
 V^{0}_{n}&=1-X_f(\tau), \text{ for } n= 0, 1, \ldots, N\\
 \frac{\partial v}{\partial y}(\tau,0)&=-X_f(\tau), \\
 V^{m}_{0} & =0, \text{ for } m=0, 1, \ldots, M\\
 X^{m}_f(0)&=1, \text{ for } m=0, 1, \ldots, M\\ 
 V^{M}_{n} & =0, \text{ for } n= 0, 1, \ldots, N.
\label{aw}  
\end{align*}
Additionally, by plugging the boundary condition of\eqref{eq11} into the main equation \eqref{eq8} and obtain the following new boundary condition
\begin{align}
 ^c_0D^\alpha _t (1-X_f(\tau))-\frac{1}{2}\sigma^2 \frac{\partial^2 v}{\partial^2 y}(\tau,0)+\frac{\sigma^2}{2}X_f(\tau)-X^{'}_f(\tau)-r=0
\end{align}
from which we deduce that
\begin{align}
\begin{split}
\frac{1}{1-\alpha}\int_0^{t}e^\frac{-\alpha(t-\tau)}{(1-\alpha)}X^{'}_f(\tau)d\tau -\frac{1}{2}\sigma^2 \frac{\partial^2 v}{\partial^2 y}(\tau,0)+\frac{\sigma^2}{2}X_f(\tau)-X^{'}_f(\tau)-r=0\label{14}
\end{split}
\end{align}
Arguing similarly as in \cite{qureshi2019new} we can see that \eqref{14} can be expressed in the following equation 
\begin{align}
\begin{split}
\frac{1}{\Delta\tau\alpha}(e^{\frac{\alpha}{1-\alpha}\Delta\tau} -1) \sum_{k=1}^{n}(X_f(t_{n+1-k},0)
-X_f(t_{n-k},0))e^{\frac{-\alpha}{1-\alpha}k\Delta\tau}\\
-\frac{1}{2}\sigma^2 \frac{\partial^2 v}{\partial^2 y}(\tau,0)+\frac{\sigma^2}{2}X_f(\tau)-X^{'}_f(\tau)-r=0.\label{14a} 
\end{split}
\end{align}
Applying the Crank-Nicolson discretization techniques to \eqref{14a} at $m=0$ and the outcome is as follows
\begin{align}
\begin{split}
\frac{1}{\Delta\tau\alpha}(e^{\frac{\alpha}{1-\alpha}\Delta\tau} -1) \sum_{k=1}^{n}(X_f(t_{n+1-k},0)-X_f(t_{n-k},0))e^{\frac{-\alpha}{1-\alpha}k\Delta\tau}\\-\frac{1}{4}\sigma^2( \frac{v^{1}_{n+1}-2v^{0}_{n+1}+v^{-1}_{n+1}}{(\Delta y)^2} +\frac{v^{1}_{n}-2v^{0}_{n}+v^{-1}_{n}}{(\Delta y)^2})\\+\frac{\sigma^2}{2}X_f(\tau)+\frac{X^{n+1}_f(\tau)-X^{n}_f(\tau)}{\Delta\tau}-r=0 \label{eq15}
\end{split}
\end{align}
From the boundary conditions \eqref{eq11}, we can obtain
\begin{align}
&\frac{1}{4}(\frac{v^{1}_{n+1}-v^{-1}_{n+1}}{\Delta y}+\frac{v^{1}_{n}-v^{-1}_{n}}{\Delta y})=-\frac{1}{2}(X^{n+1}_f(\tau)+X^{n}_f(\tau)); \label{eq16}\\
& v^{0}_{n}=1-X^{n}_f(\tau)\label{eq16a}
\end{align}
Thus, by plugging \eqref{eq16} and \eqref{eq16a} in \eqref{eq15} we have
\begin{align}
\begin{split}
\frac{1}{\Delta\tau\alpha}(e^{\frac{\alpha}{1-\alpha}\Delta\tau} -1) \sum_{k=1}^{n}(X_f(t_{n+1-k},0)-X_f(t_{n-k},0))e^{\frac{-\alpha}{1-\alpha}k\Delta\tau}\\-\frac{\sigma^2}{4(\Delta y)^2}( {2v^{1}_{n+1}+2v^{1}_{n}+2(X^{n+1}_f(\tau)+X^{n}_f(\tau))({\Delta y} +1)-4}) \\+\frac{\sigma^2}{2}X_f(\tau)+\frac{X^{n+1}_f(\tau)-X^{n}_f(\tau)}{\Delta\tau}-r=0.\label{eq17} 
\end{split}
\end{align}
Equation \eqref{eq17} can be further simplified to become 
\begin{align}
  v^{1}_{n+1}
  & =-v^{1}_{n}+\frac{2(\Delta y)^2}{\sigma^2}(\frac{e^{\frac{\alpha}{1-\alpha}\Delta\tau} -1}{\Delta\tau\alpha} \sum_{k=1}^{n}(X_f(t_{n+1-k},0)-X_f(t_{n-k},0))e^{\frac{-\alpha}{1-\alpha}k\Delta\tau}+\frac{\sigma^2}{2}X_f(\tau)+ \notag\\
  &+\frac{X^{n+1}_f(\tau)-X^{n}_f(\tau)}{\Delta\tau}-r)-(X^{n+1}_f(\tau)+X^{n}_f(\tau))({\Delta y}+1)+2 \text{  for  } n=1,2\hdots N \label{3.29}.
\end{align}
To show this, first we set $n = 0$ and  with $X_f(0,0) = 1$, we easily have that
\begin{align}
 v^{1}_{1} & =-v^{1}_{0}+\frac{2(\Delta y)^2}{\sigma^2}(\frac{\sigma^2}{2}X_f(\tau)+\frac{X^{1}_f(\tau)-X^{0}_f(\tau)}{\Delta\tau}-r)-(X^{1}_f(\tau)+X^{0}_f(\tau))({\Delta y}+1)+2 \notag\\
 & =\frac{2(\Delta y)^2}{\sigma^2}(\frac{\sigma^2}{2}+\frac{X^{1}_f(\tau)-1}{\Delta\tau}-r)-(X^{1}_f(\tau)+1)({\Delta y}+1)+2 \notag\\
 &=\frac{2(\Delta y)^2}{\sigma^2}(\frac{\sigma^2}{2}+\frac{X^{1}_f(\tau)-1}{\Delta\tau}-r)-(X^{1}_f(\tau)+1)({\Delta y}+1)+2 \label{ax1}
\end{align}
Second when  $n = 1$, and  with $X_f(0,0)=1$ then we have that
\begin{align}
 v^{1}_{2} & =-v^{1}_{1}+\frac{2(\Delta y)^2}{\sigma^2}(\frac{e^{\frac{\alpha}{1-\alpha}\Delta\tau} -1}{\Delta\tau\alpha} \sum_{k=1}^{1}(X_f(t_{n+1-k},0)-X_f(t_{n-k},0))e^{\frac{-\alpha}{1-\alpha}k\Delta\tau}\notag \\
 &+\frac{\sigma^2}{2}X_f(\tau)+\frac{X^{2}_f(\tau)-X^{1}_f(\tau)}{\Delta\tau}-r)-(X^{2}_f(\tau)+X^{1}_f(\tau))({\Delta y}+1)+2 \notag\\
 & =-v^{1}_{1}+\frac{2(\Delta y)^2}{\sigma^2}(\frac{e^{\frac{\alpha}{1-\alpha}\Delta\tau} -1}{\Delta\tau\alpha}(X_f(1,0)-X_f(0,0))e^{\frac{-\alpha}{1-\alpha}\Delta\tau}+\frac{\sigma^2}{2}X_f(\tau)\notag \\
 &+\frac{X^{2}_f(\tau)-X^{1}_f(\tau)}{\Delta\tau}-r)-(X^{2}_f(\tau)+X^{1}_f(\tau))({\Delta y}+1)+2 \notag \\
 &=-v^{1}_{1}+\frac{2(\Delta y)^2}{\sigma^2}(\frac{e^{\frac{\alpha}{1-\alpha}\Delta\tau} -1}{\Delta\tau\alpha}(X_f((1,0))-1))e^{\frac{-\alpha}{1-\alpha}\Delta\tau}+\frac{\sigma^2}{2}X_f(\tau) \notag \\
 &+\frac{X^{2}_f(\tau)-X^{1}_f(\tau)}{\Delta\tau}-r)-(X^{2}_f(\tau)+X^{1}_f(\tau))({\Delta y}+1)+2 \label{ax}
\end{align}
Collecting the results from equations \eqref{13b} to \eqref{ax} we deduce that 
\begin{align}
\mathbf{W_{n+1}}\mathbf{V_{n+1}}+\mathbf{R_{n+1}}=\mathbf{W_{n}}\mathbf{V_{n}}+\mathbf{L_{n}}+\Sigma_{n+1}
\label{axw}\end{align}
Where
\[
\mathbf{W_{n+1}}=
\begin{bmatrix} 
 B & C_{1}&0\hdots &\hdots&\hdots &\hdots \\
 0& A_{2} & B & C_{2}& 0\hdots &\hdots & \hdots &\hdots\\
0 & 0 & A_{3} & B & C_{3}&0 &\hdots &\hdots\\
0 & 0 & 0 & A_{4}& B & C_{4}&0\hdots\\
\vdots &\vdots &\vdots& \vdots & \vdots & \vdots & \ddots & \ddots\\
0 &0 &0  & 0 & 0&  0&A_{n+1}& B \\ 
\end{bmatrix}
\];
\[
\mathbf{V_{n+1}}=
\begin{bmatrix}
v^{1}_{n+1}\\ 
v^{2}_{n+1}\\
\vdots\\ 
\vdots\\ 
v^{M-1}_{n+1} 
\end{bmatrix};
\mathbf{L_{n}}=
\begin{bmatrix}
A_{n}v^{0}_{n}\\ 
0 \\
0\\
\vdots\\ 
\vdots\\ 
A_{n}v^{M}_{n} 
\end{bmatrix};
\mathbf{R_{n+1}}=
\begin{bmatrix}
A_{n+1}v^{0}_{n+1}\\ 
0 \\
0\\
\vdots\\ 
\vdots\\ 
A_{n+1}v^{M}_{n+1} 
\end{bmatrix};
\];
\[
\Sigma_{n+1}=\begin{bmatrix} 
\sum_{k=1}^{n}(e^{\frac{-\alpha}{1-\alpha}(n+1-k)\Delta\tau}-e^{\frac{-\alpha}{1-\alpha}(n-k)\Delta\tau})v^{1}_{k}\\
 \sum_{k=1}^{n}(e^{\frac{-\alpha}{1-\alpha}(n+1-k)\Delta\tau}-e^{\frac{-\alpha}{1-\alpha}(n-k)\Delta\tau})v^{2}_{k} \\
\sum_{k=1}^{n}(e^{\frac{-\alpha}{1-\alpha}(n+1-k)\Delta\tau}-e^{\frac{-\alpha}{1-\alpha}(n-k)\Delta\tau})v^{3}_{k}\\
\sum_{k=1}^{n}(e^{\frac{-\alpha}{1-\alpha}(n+1-k)\Delta\tau}-e^{\frac{-\alpha}{1-\alpha}(n-k)\Delta\tau})v^{4}_{k} \\
\vdots \\
\sum_{k=1}^{n}(e^{\frac{-\alpha}{1-\alpha}(n+1-k)\Delta\tau}-e^{\frac{-\alpha}{1-\alpha}(n-k)\Delta\tau})v^{M-1}_{k}\\ 
\end{bmatrix}\]
Hence \eqref{axw} becomes
\begin{align}
\mathbf{V_{n+1}}=\mathbf{W^{-1}_{n+1}}(\mathbf{W_{n}}\mathbf{V_{n}}+\mathbf{L_{n}}+\Sigma_{n+1}-\mathbf{R_{n+1}})
\label{axw1}\end{align}
The Crank-Nicolson discretization of the Caputo fractional differential equation of APO becomes equation \eqref{axw1}. Regarding the price and the free boundary condition, it was not linear at every time step.\\
By evaluating the $(n+1)$ term of \eqref{3.29}, we assert that using the scheme \eqref{13a} at $m=1$, $X^{n+1}_f$ can be expressed as,
\begin{align}
\begin{split}
X^{n+1}_f(\tau)=\frac{\Omega_1}{\Omega_2}
\end{split}
\end{align}
where
\begin{align}
\Omega_1= &\sum_{k=1}^{n}(v(t_{n+1-k},y_1)-v(t_{n-k},y_1))e^{\frac{-\alpha}{1-\alpha}k\Delta\tau}-\frac{\Delta\tau}{e^{\frac{\alpha}{1-\alpha}\Delta\tau} -1} (\frac{\alpha\sigma^2}{4(\Delta y)^2}(V^{2}_{n+1}+V^{0}_{n+1}+V^{2}_{n}+V^{0}_{n})\notag\\&+(1+r\frac{\Delta y^2}{\sigma^2})(\frac{e^{\frac{\alpha}{1-\alpha}\Delta\tau} -1}{\Delta\tau\alpha}\sum_{k=1}^{n}(X_f(t_{n+1-k},0)-X_f(t_{n-k},0))e^{\frac{-\alpha}{1-\alpha}k\Delta\tau}\notag\\&+\frac{X^{n}_f(\tau)}{\Delta\tau}-r)+X^{n}_f(\tau)({\Delta y}+1)+2) 
\end{align}\\
\begin{align}
 \Omega_2=\frac{\Delta\tau\alpha}{2(e^{\frac{\alpha}{1-\alpha}\Delta\tau} -1)}(\frac{V^{2}_{n+1}-V^{0}_{n+1}+V^{2}_{n}-V^{0}_{n}}{2\Delta\tau \Delta y X^{n}_f(\tau)}-(\frac{\sigma^2}{(\Delta y)^2}+r)(\frac{2(\Delta y)^2}{\Delta\tau\sigma^2}-1))
\end{align}\\
\begin{align}
\begin{split}
X^{n+1}_f(\tau)&=\frac{%
 \splitdfrac{\splitdfrac{\sum_{k=1}^{n}(v(t_{n+1-k},y_1)-v(t_{n-k},y_1))e^{\frac{-\alpha}{1-\alpha}k\Delta\tau}-\frac{\Delta\tau}{e^{\frac{\alpha}{1-\alpha}\Delta\tau} -1} (\frac{\alpha\sigma^2}{4(\Delta y)^2}(V^{2}_{n+1}+V^{0}_{n+1}+V^{2}_{n}+V^{0}_{n})}%
                {+\frac{\alpha}{4\Delta\tau(\Delta y)}(r-\frac{\sigma^2}{2}-1)(V^{2}_{n+1}-V^{0}_{n+1}+V^{2}_{n}-V^{0}_{n})+(1+r\frac{\Delta y^2}{\sigma^2})(\frac{e^{\frac{\alpha}{1-\alpha}\Delta\tau} -1}{\Delta\tau\alpha}}}%
                          { \sum_{k=1}^{n}(X_f(t_{n+1-k},0)-X_f(t_{n-k},0))e^{\frac{-\alpha}{1-\alpha}k\Delta\tau}+\frac{X^{n}_f(\tau)}{\Delta\tau}-r)+ X^{n}_f(\tau)({\Delta y}+1)+2)}}%
             {\frac{\Delta\tau\alpha}{2(e^{\frac{\alpha}{1-\alpha}\Delta\tau} -1)}\left(\frac{V^{2}_{n+1}-V^{0}_{n+1}+V^{2}_{n}-V^{0}_{n}}{2\Delta\tau \Delta y X^{n}_f(\tau)}-\left[\frac{\sigma^2}{(\Delta y)^2}+r\right]\left[\frac{2(\Delta y)^2}{\Delta\tau\sigma^2}-1\right]\right)}\label{x_f}
\end{split}
\end{align}
\red{Please explain why is the denominator cannot be zero}\\

After the expressions in \eqref{x_f}, the value of the term $X^{n+1}_f$ can be replaced in \eqref{13a}, \eqref{3.29}, and \eqref{eq11} to obtain the values of $V^{m}_{n+1}, 0\leq m \leq M$. Therefore, in the accomplished section, we develop the numerical scheme with a novel algorithm for equations \eqref{eq4} to \eqref{eq7} for any $n = 0 \cdots \cdots N-1$.\\
This section has focused on Crank-Nicolson's discretization of APO. The next part of this paper will discuss the consistency and stability of discretization.\\
\section{Positivity and Monotonicity} 
This section will demonstrate the decreasing and positivity of the free boundary and decreasing of numerical option price. Let's begin this part by showing coefficients $A_{n},B \text{   and   } C_{n} $  of the scheme \eqref{13b} are non-negative under appropriate conditions of the step size discretization $\Delta \tau \text{   and   }\Delta y $.
\begin{Lem}(Company et al.\cite{egorova2014solving})\label{lemma1} 
If $\Delta \tau \text{   and   }\Delta y $ satisfy the following the following inequalities:
 \begin{align}
\Delta y \leq \frac{\sigma^2 \Delta \tau}{\vert r-\frac{\sigma^2}{2}\vert},r \not=\frac{\sigma^2}{2} \\
\Delta \tau \leq \frac{\Delta y^2}{r\Delta y^2+\sigma^2} \label{post}
 \end{align}
 then the coefficients $A_{n},B \text{   and   } C_{n} $ are non-negative. If $ r=\frac{\sigma^2}{2} $,then the non-negativity of these coefficients is  verified under the condition \eqref{post}.
\end{Lem}
\begin{Thm}(Company et al. \cite{egorova2014solving})
Let $(V^{m}_{n}, X^{n}_{f})$ be the computed numerical solution and $X^{n}_{f}$ be defined by equations \eqref{x_f}. Under the hypothesis of Lemma \eqref{lemma1}, for sufficiently small values of  $\Delta y $, we have:  
\begin{itemize}
  \item $X^{n}_{f}$ is positive and non-increasing monotone for $n = 0, 1,\hdots, N$;
  \item the vectors $V^{m}_{n}$  have positive components for $n = 0, 1,\hdots, N$;
  \item the vectors $V^{m}_{n}$  are non-increasing monotone with respect to $m$ for each fixed $n = 0, 1,\hdots, N.$
\end{itemize}
\end{Thm}
\section{Consistency and Stability}
The consistency and stability of the Crank-Nicolson method of Caputo Fabrizio American put option \eqref{13a} discussed in this section.\\ 
\subsection{Consistency}
\begin{Def}
Let $\mathcal{F}(V^{m+1}_n,X^{n+1}_{f})$ be the difference operator approximating $\mathcal{L}(V,X_{f})$. The
local truncation error $ \mathcal{T}^{m+1}_{n}(\bar{V},\bar{X_{f}})$ is given by the leading terms of the Taylor expansion of $\mathcal{F}(V^{m+1}_n,X^{n+1}_{f})=0$,where $u$ satisfies $\mathcal{L}(V,X_{f})=0 $.That  means 
\end{Def}
\begin{align}
\mathcal{T}^{m+1}_{n}(\bar{V},\bar{X_{f}})=\mathcal{F}(V^{m+1}_n,X^{n+1}_{f}),\text{  and} \notag \\
 \lim_{(\Delta \tau,\Delta y)\rightarrow (0,0)}\mathcal{T}^{m+1}_{n}(\bar{V},\bar{X_{f}})=0.
\end{align}
Let's start by choosing an arbitrary location$(y,\tau) \in [0,4E]\times [0,T]$ the mesh point $(y^{m},\tau^{n+1})$ to study the consistency of Caputo Fabrizio PDE \eqref {eq4} and its numerical scheme  \eqref{13a}  employing Crank-Nicolson finite difference  method. Let $\bar{V}^{m+1}_{n}=V(\tau^{n},y_{m+1})$ be the exact solution for  Caputo Fabrizio PDE, and $\bar{X}^{n+1}_{f}=X_{f}(\tau^{n+1})$ the free boundary respectively. That means
\begin{align*}
\mathcal{L}(V,X_{f})=^c_0D^\alpha _t v(\tau,y) -\frac{1}{2}\sigma^2 \frac{\partial^2 v}{\partial^2 y}+(r-\frac{\sigma^2}{2})\frac{\partial v}{\partial y}-\frac{1}{X_f(\tau)}\frac{dX_f(\tau)}{d\tau}\frac{\partial v}{\partial y}-rv=0\\
\end{align*}
and 
\begin{align*}
  \mathcal{F}(V^{m+1}_n,X^{n+1}_{f})=&\frac{1}{\Delta\tau\alpha}(e^{\frac{\alpha}{1-\alpha}\Delta\tau}-1) \sum_{k=1}^{n}(v(t_{n+1-k},y_m)-v(t_{n-k},y_m))e^{\frac{-\alpha}{1-\alpha}k\Delta\tau} \notag \\&=\frac{1}{2}(\frac{1}{2}\sigma^2 \frac{v^{m+1}_{n+1}-2v^{m}_{n+1}+v^{m-1}_{n+1}}{(\Delta y)^2} +(\frac{\sigma^2}{2}-r)\frac{v^{m+1}_{n+1}-v^{m-1}_{n+1}}{2\Delta y}\notag \\&+\frac{X^{n+1}_f(\tau)-X^{n}_f(\tau)}{\Delta\tau X^{n}_f(\tau)}\frac{v^{m+1}_{n+1}-v^{m-1}_{n+1}}{2\Delta y}+rv^{m+1}_{n+1}+\frac{1}{2}\sigma^2 \frac{v^{m+1}_{n}-2v^{m}_{n}+v^{m-1}_{n}}{(\Delta y)^2}\notag \\&+(\frac{\sigma^2}{2}-r)\frac{v^{m+1}_{n}-v^{m-1}_{n}}{2\Delta y}+\frac{X^{n+1}_f(\tau)-X^{n}_f(\tau)}{\Delta\tau X^{n}_f(\tau)}\frac{v^{m+1}_{n}-v^{m-1}_{n}}{2\Delta y}+rv^{m}_{n}). \label{3.3a}\\   
\end{align*}
\begin{Thm}
The Crank-Nicolson Caputo Fabrizio fractional derivatives scheme method defined by Equations \eqref{13a} is consistent with the fixed domain model  \eqref{eq4}  of order $ \mathcal{O}(\Delta \tau)+\mathcal{O}((\Delta y)^2) $.
\end{Thm}
\begin{proof}
We implement Taylor's expansion similar to the techniques used in \cite{fazio2021front}, taking into account the continuous PD of orders two and four in space and time, respectively to show consistency of \eqref{eq8}. First we will find the consistency of the Caputo Fabrizio differential equation in some parts of  \eqref{eq8}. By following the lines of \cite{qureshi2019new} we can compute that the first order time-fractional derivatives of scheme equation \eqref{eq8} of the Caputo Fabrizio differential equation which is given by
\begin{align}
\begin{split}
^c_0D^\alpha _t v(\tau=t_n,y=y_m) =\frac{1}{\Delta\tau\alpha}(e^{\frac{\alpha}{1-\alpha}\Delta\tau} -1) \sum_{k=1}^{n}(v(t_{n+1-k},y_m)-v(t_{n-k},y_m))e^{\frac{-\alpha}{1-\alpha}k\Delta\tau}.
\end{split}
\end{align}
with truncation error of\\
\begin{align}
\mathbf{e_1}=\mathcal{O}(\Delta\tau)\label{tr1}
\end{align}
Second, let us find the consistency of each of the last three right-hand side equations of \eqref{eq8} one by one, using Taylor's expansion up to order two  and four in time and space, respectively.
First, let us find the truncation error of the first term of the right-hand side of \eqref{eq8}
\begin{align}
\frac{1}{2}\sigma^2 \frac{\partial^2 v}{\partial^2 y}&=\frac{1}{4}\sigma^2 (\frac{v^{m+1}_{n+1}-2v^{m}_{n+1}+v^{m-1}_{n+1}}{(\Delta y)^2} + \frac{v^{m+1}_{n}-2v^{m}_{n}+v^{m-1}_{n}}{(\Delta y)^2})\notag\\&=\frac{1}{2}\sigma^2(\frac{\partial^2 v}{\partial^2 y}+(\Delta y)^2\frac{1}{12}\frac{\partial^4 v}{\partial^4 y}+\mathcal{O}(\Delta y)^3).\label{ee1}  
\end{align}
Therefore the error term of \eqref{ee1} becomes
\begin{align}
\mathbf{e_2}=(\Delta y)^2\frac{1}{12}\frac{\partial^4 v}{\partial^4 y}+\mathcal{O}(\Delta y)^3\label{tr2}
\end{align}
Second, let us find the truncation error of the second expression right-hand side of \eqref{eq8}
\begin{align}
(\frac{\sigma^2}{2}-r)\frac{\partial v}{\partial y}=\frac{1}{2}(\frac{\sigma^2}{2}-r)(\frac{v^{m+1}_{n+1}-v^{m-1}_{n+1}}{2\Delta y}+\frac{v^{m+1}_{n}-v^{m-1}_{n}}{2\Delta y})\notag \\=
 \frac{1}{2}(\frac{\sigma^2}{2}-r)(\frac{\partial v}{\partial y}+(\Delta y)^2\frac{1}{6}\frac{\partial^3 v}{\partial^3 y} +\mathcal{O}(\Delta y)^4).\label{ee2} 
\end{align}
Therefore the error term of \eqref{ee2} becomes
\begin{equation}
\mathbf{e_3}=(\Delta y)^2\frac{1}{6}\frac{\partial^3 v}{\partial^3 y}+\mathcal{O}(\Delta y)^4\label{tr3}
\end{equation}
Next, let us find the truncation error of the third  expression of the right-hand side of \eqref{eq8}
\begin{align}
\frac{1}{X_f(\tau)}\frac{dX_f(\tau)}{d\tau}\frac{\partial v}{\partial y}&=\frac{1}{2}(\frac{X^{n+1}_f(\tau)-X^{n}_f(\tau)}{\Delta\tau X_f(\tau)}(\frac{v^{m+1}_{n+1}-v^{m-1}_{n+1}}{2\Delta y}+\frac{v^{m+1}_{n}-v^{m-1}_{n}}{2\Delta y})\notag\\&=\frac{1}{X_f(\tau)}(\frac{dX_f(\tau)}{d\tau}+(\Delta \tau)^2\frac{d^2X_f(\tau)}{d\tau^2}+\mathcal{O}((\Delta \tau)^2))(\frac{\partial v}{\partial y}+(\Delta y)^2\frac{1}{6}\frac{\partial^3 v}{\partial^3 y}+\mathcal{O}(\Delta y)^4)\notag\\&=\frac{1}{X_f(\tau)}\frac{dX_f(\tau)}{d\tau}\frac{\partial v}{\partial y}+\frac{1}{X_f(\tau)}\frac{dX_f(\tau)}{d\tau}((\Delta y)^2\frac{1}{6}\frac{\partial^3 v}{\partial^3 y} +\mathcal{O}(\Delta y)^4)
\notag\\&+\frac{1}{X_f(\tau)}((\Delta \tau)^2\frac{d^2X_f(\tau)}{d\tau^2}+\mathcal{O}((\Delta \tau)^2))(\frac{\partial v}{\partial y}+(\Delta y)^2\frac{1}{6}\frac{\partial^3 v}{\partial^3 y} +\mathcal{O}(\Delta y)^4).\label{ee3}
\end{align}
Therefore the error term of \eqref{ee3} becomes
\begin{align}
\mathbf{e_4}&=\frac{1}{X_f(\tau)}\frac{dX_f(\tau)}{d\tau}((\Delta y)^2\frac{1}{6}\frac{\partial^3 v}{\partial^3 y} +\mathcal{O}(\Delta y)^4)\notag\\&
+\frac{1}{X_f(\tau)}((\Delta \tau)^2\frac{d^2X_f(\tau)}{d\tau^2}+\mathcal{O}((\Delta \tau)^2))(\frac{\partial v}{\partial y}+(\Delta y)^2\frac{1}{6}\frac{\partial^3 v}{\partial^3 y} +\mathcal{O}(\Delta y)^4)\label{tr4}
\end{align}
Lastly, we can find the truncation error of the last term of the right-hand side of \eqref{eq8}
\begin{align}
rv = \frac{1}{2}(rv^{m}_{n+1}+rv^{m}_{n})\label{tr5}
\end{align}
Hence,the local truncation error of \eqref{eq8}  using \eqref{tr1} to \eqref{tr5} is provided by
\begin{align}
 \mathcal{T}^{m+1}_{n}(\bar{V},\bar{X_{f}})&=\mathbf{e_1}+\mathbf{e_2}+\mathbf{e_3}+\mathbf{e_4}
 =\mathcal{O}(\Delta \tau)+\mathcal{O}((\Delta y)^2) \label{tr2121} 
\end{align}
Now let us find the local truncation error of the boundary conditions of equations \eqref{eq7} and \eqref{eq17}
\begin{align}
&\mathcal{L}_1(V,X_{f})=v(\tau,0)-1+X_f(\tau)=0 \notag\\	
&\mathcal{L}_2(V,X_{f})=\frac{\partial v}{\partial y}(\tau,0)+X_f(\tau)=0 \notag\\
&\mathcal{L}_3(V,X_{f})=^c_0D^\alpha _t (1-X_f(\tau))-\frac{1}{2}\sigma^2 \frac{\partial^2 v}{\partial^2 y}(\tau,0)+\frac{\sigma^2}{2}X_f(\tau)-X^{'}_f(\tau)-r=0.\label{tr12a}
 \end{align}
The numerical boundary scheme is given by
\begin{align}
&\mathcal{F}_1(V^{0}_{n+1},X^{n+1}_{f})=V^{0}_{n+1}-1+X^{n+1}_{f} \notag\\
&\mathcal{F}_2(V^{0}_{n+1},X^{n+1}_{f})=\frac{v^{1}_{n+1}-v^{-1}_{n+1}+v^{1}_{n}-v^{-1}_{n}}{2\Delta y}+X^{n+1}_{f}\notag\\
&\mathcal{F}_3(V^{0}_{n+1},X^{n+1}_{f})=\frac{1}{\Delta\tau\alpha}(e^{\frac{\alpha}{1-\alpha}\Delta\tau} -1)\sum_{k=1}^{n}(X_f(t_{n+1-k},0)-X_f(t_{n-k},0))e^{\frac{-\alpha}{1-\alpha}k\Delta\tau}\notag-\\&\frac{1}{4}\sigma^2( \frac{v^{1}_{n+1}-2v^{0}_{n+1}+v^{-1}_{n+1}}{(\Delta y)^2} +\frac{v^{1}_{n}-2v^{0}_{n}+v^{-1}_{n}}{(\Delta y)^2})+\frac{\sigma^2}{2}X_f(\tau)+\frac{X^{n+1}_f(\tau)-X^{n}_f(\tau)}{\Delta\tau}-r=0
\label{12}.
\end{align}
By following the same steps as in \eqref{tr1} to \eqref{tr5} for \eqref{tr12a} and  \eqref{12}, we obtained a local truncation error
\begin{align}
\mathcal{O}(\Delta \tau)+\mathcal{O}((\Delta y)^2) \label{13} 
\end{align}
That means
\begin{align*}
&\mathcal{T}_1^{0}(\bar{V},\bar{X_{f}}) = \mathcal{F}_1(V^{0}_{n+1},X^{n+1}_{f})=0 \notag\\
&\mathcal{T}_2^{0}(\bar{V},\bar{X_{f}})=(\Delta y)^2\frac{1}{3}\frac{\partial^3 v}{\partial^3 y}+\mathcal{O}(\Delta y)^4 \notag\\
&\mathcal{T}_3^{0}(\bar{V},\bar{X_{f}})=\mathcal{O}(\Delta \tau)+\mathcal{O}((\Delta y)^2).
\end{align*}
As a result of \eqref{tr2121} and \eqref{13}, the Crank-Nicolson method provided by \eqref{13a} is compatible with the fixed domain model \eqref{eq4} of order 
\begin{align*}
 \mathcal{O}(\Delta \tau)+\mathcal{O}((\Delta y)^2).   
\end{align*}
\end{proof}
\subsection{Stability}
The boundedness of the solution of a finite difference scheme the partial differential equation affects its stability. For more details we refer for instance to the works in \cite{hoffman2018numerical,hoffmann2000computational,smith1985numerical}. Namely, stability of a partial differential equation (PDE) is determined by the boundedness of its solution when using a finite difference scheme. Note that Various methods, such as the discrete perturbation method, the Von Neumann method, and the matrix method \cite{ hoffmann2000computational}. are employed to analyze the stability of numerical schemes for PDEs. In this paper, we apply the Von Neumann technique \cite{mphephu2013numerical,smith1985numerical} to analyze the stability of time fractional PDEs using the Crank-Nicolson numerical scheme.

Von Neumann analysis which is an application of Fourier analysis, is employed to analyze a finite linear scheme with constant coefficients, see the results by Strikwerda in \cite{strikwerda2004finite} and Smith in \cite{smith1985numerical}. This limitation is addressed in the case of nonlinear problems by linearizing and fixing the coefficients, thereby considering the problem locally, as discussed by Aslak and Ragnar  \cite{mphephu2013numerical}.

The stability of a partial differential equation (PDE) is tested using Von Neumann analysis as follows: first, if the PDE is not already linear, it is converted to a linear form. Then, in a one-step finite numerical scheme, the term $U_{n}$ is replaced with $\lambda e^{imy}$ for every value of $n$ and $m$. Finally, we solve for $\lambda$. The term $\lambda$ is referred to as the amplification factor of the finite numerical scheme, and it determines the stability of the scheme. According to Hoffmann and Chiang \cite{hoffmann2000computational} and Smith \cite{smith1985numerical} if $ \lambda \leq 1 $  then we assert that finite difference scheme is stable.
\begin{Thm}
The Crank-Nicolson Caputo Fabrizio fractional derivatives scheme method defined by Equations \eqref{13b} is unconditionally stable.
\end{Thm}
\begin{proof}
The non-linear nature of the model in \eqref{eq4} arises in the presence of 
\begin{align}
\begin{split}
\frac{1}{X_f(\tau)}\frac{dX_f(\tau)}{d\tau} \label{st1} 
\end{split}
\end{align}
Then we linearize and approximate \eqref{st1} by  considering $X_f(\tau)$ and its derivatives in the stability analysis.
\begin{align}
\begin{split}
\frac{1}{X_f(\tau)}\frac{dX_f(\tau)}{d\tau}=\frac{X^{n+1}_f(\tau)-X^{n}_f(\tau)}{\Delta\tau X^{n+1}_f(\tau)}
\end{split}
\end{align}
that assumes non positive values because $X^n_{f}(\tau)$ is positive and not increasing.\\
Let 
\begin{align*}
&v^{m}_{n}=e^{at} e^{iby} \notag\\
&v^{m+1}_{n+1}=e^{a(t+\Delta\tau)} e^{i(b(y+\Delta y)}=\lambda e^{at} e^{ib(y+\Delta y)} \notag\\
&v^{m-1}_{n}=e^{at} e^{ib(y-\Delta y)}=\lambda e^{ib(y-\Delta y)}\notag\\
&v^{m+1}_{n}=e^{at} e^{ib(y+\Delta y)} \notag\\
&v^{m-1}_{n+1}=e^{a(t+\Delta\tau)} e^{ib(y-\Delta y)}=\lambda e^{at} e^{ib(y-\Delta y)}\notag\\
&v^{m}_{n+1}=e^{a(t+\Delta\tau)} e^{iby}=\lambda e^{at} e^{iby}\notag\\
&v^{m}_{n+1-k}=e^{a(t+(1-k)\Delta\tau)} e^{iby}=\lambda e^{at\Delta\tau}e^{-ak\Delta\tau} e^{iby}\notag\\
&v^{m}_{n-k}=e^{a(t-k)\Delta\tau} e^{iby}= e^{at\Delta\tau} e^{-ak\Delta\tau} e^{iby} \notag \text{~~and~~~} \\
& \lambda=\frac{v^{m}_{n+1}}{v^{m}_{n}}=e^{a\Delta\tau}
\end{align*}
be a Fourier component and where $b, a \in \mathbb{R}$ with $a \neq 0$ and $b \neq 0$.\\
Hence \eqref{13b} becomes
\begin{align}
 \sum_{k=1}^{n}(e^{a(t+(1-k)\Delta\tau)}-e^{a(t-k\Delta\tau)} )e^{iby}e^{\frac{-\alpha}{1-\alpha}k\Delta\tau}=&A_{n} e^{a(t+\Delta\tau)} e^{i(b(y+\Delta y)}+B e^{a(t+\Delta\tau)} e^{iby}+ C_{n} e^{a(t+\Delta\tau)} e^{ib(y-\Delta y)} \notag \\&+ A_{n} e^{at} e^{ib(y+\Delta y)}+ B e^{at} e^{iby} +C_{n} e^{at} e^{ib(y-\Delta y)} 
\end{align}
Thus
\begin{align}
 \sum_{k=1}^{n}(\lambda - 1)e^{-k\Delta\tau(a+\frac{\alpha}{1-\alpha})}=&\lambda A_{n}  e^{ib \Delta y} + \lambda B +\lambda C_{n} e^{-ib \Delta y}+ A_{n} e^{ib \Delta y }+ B  + C_{n}  e^{-ib\Delta y} 
\end{align}
This implies 
\begin{align}
\begin{split}
 \sum_{k=1}^{n}(\lambda -1)e^{-k\Delta\tau(\frac{\alpha}{1-\alpha}+a)}=(\lambda+1) (A_{n}  e^{ib \Delta y}+ B + C_{n}  e^{-ib\Delta y} )
\end{split}
\end{align}
From this we deduce that 
\begin{align}
 (\lambda -1)\sum_{k=1}^{n}e^{-k\Delta\tau(\frac{\alpha}{1-\alpha}+a)}&=(\lambda+1)\frac{\Delta\tau\alpha}{2(e^{\frac{\alpha}{1-\alpha}\Delta\tau} -1)}(\frac{-\sigma^2}{(\Delta y)^2}(1- \cos(b\Delta y))\notag\\&+\frac{i \sin(b\Delta y)}{\Delta y}(r-\frac{\sigma^2}{2})
+\frac{i \sin(b\Delta y)(X^{n+1}_f(\tau)- X^{n}_f(\tau))}{\Delta y\Delta\tau X^{n}_f(\tau)}-r)\notag\\&
=(\lambda + 1)\frac{\Delta\tau\alpha}{2(e^{\frac{\alpha}{1-\alpha}\Delta\tau} -1)}(\frac{-2\sigma^2}{(\Delta y)^2} \sin^{2}(\frac{b\Delta y)}{2})\notag\\&+\frac{i \sin(b\Delta y)}{\Delta y}(r-\frac{\sigma^2}{2})
+\frac{i \sin(b\Delta y)(X^{n+1}_f(\tau)- X^{n}_f(\tau))}{\Delta y\Delta\tau X^{n}_f(\tau)}-r) \label{stt}
\end{align}
Since the left-hand side of \eqref{stt} is a pure real  and after further simplification it becomes
\begin{align}
\begin{cases}
 \lambda= \Large {\frac{\frac{2\left(e^{\frac{\alpha}{1-\alpha}\Delta\tau}-1\right)}{\Delta\tau\alpha}\sum_{k=1}^{n}e^{-k\Delta\tau(\frac{\alpha}{1-\alpha}+a)}-\frac{2\sigma^2}{(\Delta y)^2} \sin^{2}(\frac{b\Delta y}{2})-r}{\frac{2\left(e^{\frac{\alpha}{1-\alpha}\Delta\tau}-1\right)}{\Delta\tau\alpha }\sum_{k=1}^{n}e^{-k\Delta\tau(\frac{\alpha}{1-\alpha}+a)}+\frac{2\sigma^2}{(\Delta y)^2} \sin^{2}(\frac{b \Delta y}{2})+r}} \\\\ \frac{i \sin(b\Delta y)}{\Delta y}(r-\frac{\sigma^2}{2})+\frac{i \sin(b \Delta y)(X^{n+1}_f(\tau)- X^{n}_f(\tau))}{\Delta y\Delta \tau X^{n}_f(\tau)}=0 
\end{cases}
\end{align}
With the above system in hand, the main idea here is to show that
\begin{align}
\lVert \lambda \rVert \leq 1
\end{align}
From 
\begin{align*}
\frac{1}{1-\alpha}>1
\end{align*}
which implies that 
\begin{align}
e^{\frac{\alpha}{1-\alpha}}=e^{\frac{1}{1-\alpha}}e^{-1} > 1 \notag\\
e^{{\frac{\alpha}{1-\alpha}}\Delta\tau}> 1. \label{stabl}
\end{align}
By using \eqref{stabl} it is obvious concluding 
\begin{align}
\lVert \lambda \rVert= \bigg\lVert \frac{\frac{2\left(e^{\frac{\alpha}{1-\alpha}\Delta\tau}-1\right)}{\Delta\tau\alpha}\sum_{k=1}^{n}e^{-k\Delta\tau(\frac{\alpha}{1-\alpha}+a)}-\frac{2\sigma^2}{(\Delta y)^2} \sin^{2}(\frac{b\Delta y}{2})-r}{\frac{2\left(e^{\frac{\alpha}{1-\alpha}\Delta\tau}-1\right)}{\Delta\tau\alpha}\sum_{k=1}^{n}e^{-k\Delta\tau(\frac{\alpha}{1-\alpha}+a)}+\frac{2\sigma^2}{(\Delta y)^2}\sin^{2}(\frac{b\Delta y}{2})+r}\bigg\rVert < 1
\end{align}
With the above inequality, we can conclude that the Crank-Nicolson scheme for the  American put option with Caputo-Fabrizio fractional order derivative is unconditionally stable.
\end{proof}
\section{Numerical experiments and comparisons}
\subsection{Algorithm with MATLAB}
In this section we can write the algorithm of section two with MATLAB code.The algorithm below is obtained by incorporating all the equations defined in section two:\\
\begin{enumerate}
\item Input$~~\sigma, r, E, T, M, \mu~~ and~~ Y $;
\item Define the grid $(y_{m}, \Delta y, t_{n}, \Delta t)$;
\item For $m =0, 1, 2,\hdots, M $ do $V^{0}_{m}=0$ and set $X^{0}_f=1$;\\
Define the following
\begin{align*}
&\vartheta _{k}=e^{\frac{-\alpha}{1-\alpha}k\Delta\tau}; \notag\\
&\iota_{k,m}=v(t_{n+1-k},y_m)-v(t_{n-k},y_m);\notag\\
&\phi_{k}=X_f(t_{n+1-k},0)-X_f(t_{n-k},0);\notag\\
&\theta=\frac{\Delta\tau\alpha\sigma^2}{4(\Delta y)^2[\vartheta _{-1} -1]};\notag\\
&\beta=\frac{\Delta\tau\alpha}{4(\Delta y)[\vartheta _{-1} -1]}(r-\frac{\sigma^2}{2});\notag\\
&\omega_{n}=\frac{\Delta\tau\alpha}{e^{\frac{\alpha}{1-\alpha}\Delta\tau} -1}[\frac{X^{n+1}_f(\tau)- X^{n}_f(\tau)}{4\Delta y\Delta\tau X^{n}_f(\tau)}]\notag\\
&\Omega_1= \sum_{k=1}^{n}\iota_{k,1}\vartheta _{k}-\frac{\Delta\tau}{\vartheta _{-1} -1} (\frac{\alpha\sigma^2}{4(\Delta y)^2}(V^{2}_{n+1}+V^{0}_{n+1}+V^{2}_{n}+V^{0}_{n})\notag\\+&(1+r\frac{\Delta y^2}{\sigma^2})(\frac{\vartheta _{-1} -1}{\Delta\tau\alpha}\sum_{k=1}^{n}\phi_{k}\vartheta _{k}+\frac{X^{n}_f(\tau)}{\Delta\tau}-r)+X^{n}_f(\tau)({\Delta y}+1)+2)\notag \\
&\Omega_2=\frac{\Delta\tau\alpha}{2(\vartheta _{-1} -1)}(\frac{V^{2}_{n+1}-V^{0}_{n+1}+V^{2}_{n}-V^{0}_{n}}{2\Delta\tau \Delta y X^{n}_f(\tau)}-(\frac{\sigma^2}{(\Delta y)^2}+r)(\frac{2(\Delta y)^2}{\Delta\tau\sigma^2}-1))
\end{align*} 
\item Then compute
\begin{align*}
&A_{n}=\theta+\beta+\omega_{n};\notag\\
&C_{n}=\theta-\beta-\omega_{n}; \notag\\
&B=\frac{-\Delta\tau\alpha}{2[\vartheta _{-1} -1]}(\frac{\sigma^2}{(\Delta y)^2}+r)\notag\\
&X^{n+1}_f(\tau)=\frac{\Omega_1}{\Omega_2};\notag\\
& v^{0}_{n}=1-X^{n}_f(\tau);\notag\\
&\sum_{k=1}^{n}\iota_{k,m}\vartheta _{k} =A_{n} v^{m+1}_{n+1}+ B v^{m}_{n+1}+ C_{n} v^{m-1}_{n+1}+ A_{n} v^{m+1}_{n}+ B v^{m}_{n} +C_{n} v^{m-1}_{n} ~~\text{and};\notag\\
&\text{end, set~~} v^{M}_{n}= 0, ~~\text{end}
\end{align*}
\end{enumerate}
We implement the above algorithm in MATLAB to get numerical results as follows.
\subsection{Numerical Experiments}
In this section, we have to present a numerical test of the Crank-Nicolson finite difference scheme. To this end, the APO problem \eqref{eq4} to \eqref{eq7} is taken into account under the following conditions:\\
 $r= 0.1, \sigma = 0.2, E=T=1$\\
Our first and foremost aim is to investigate numerically the appropriate value of $Y$.To achieve this straightforward, we can monitor and compare the numerical result obtained $X^N_f~~ at~~t=T$ for different values of gird steps by fixing $\mu=20 $ with $Y=1,Y=2 \text{~and~ }Y=4$, and decide the appropriate value of $Y$.\\
Then, we compare and contrast our findings with those of Nielsen et al.\cite{nielsen2002penalty} and Riccardo.F \cite{fazio2020american}.
By refining the mesh, we improved numerical accuracy. Recall that the Crank-Nicolson difference scheme is first-order in time and second-order accurate in space for both the field variable and the free boundary value. Using the above result and keeping the constant $\mu$, we perform a mesh refinement, and consequently, we wind up with a second-order truncation error in time.

\clearpage
\bibliographystyle{plain}
\bibliography{Ref_Abera}
\nocite{*}
\end{document}